\documentclass[12pt, twoside, draft]{article}
\usepackage[cp1251]{inputenc}
\newcommand{\CopyName}{A.\ A.\ Shalukhina} 
\newcommand{\Year}{200?} 
\newcommand{\Volume}{?} 
\newcommand{\Number}{?} 
\newcommand{\Pages}{?--?} 
\usepackage[english,russian,ukrainian]{babel}
\usepackage{decor}
\renewcommand{\refname}{\refnam}
\English
\vs
\newcommand{\tit}{ON THE EXTENSION OF THE REVERSE H\"OLDER INEQUALITY\\ FOR POWER FUNCTIONS ON THE REAL AXIS} 
\date{}

\begin{document}

\English
\hskip50pt
\begin{center} \renewcommand{\baselinestretch}{1.3}\bf {\tit} \end{center}
\begin{center} \textsc {\CopyName} \end{center}

\vspace{20pt plus 0.5pt} {\abstract{\textsc{Abstract}. We consider the class of all non-negative on $\mathbb{R_+}$ functions such that each of them satisfies the Reverse H\"older Inequality uniformly over all intervals with some constant the minimum value of which can be regarded as the corresponding ``norm'' of a function. We compare this ``norm'' with the ``norm'' of an even extension of a function from $\mathbb{R_+}$ on $\mathbb{R}.$ In this paper the upper estimate for the ratio of such ``norms'' has been obtained. In the particular case of power functions on $\mathbb{R_+}$ the precise value of the increase of the ``norm'' of its even extension is given. This value is the lower estimate for the analogous one in the case of arbitrary functions. It has been shown that the obtained upper and lower estimates for the general case are asymptotically sharp.


}} \vsk

\renewcommand{\refname}{\refnam}
\renewcommand{\proofname}{\ifthenelse{\value{lang}=0}{Proof}{\ifthenelse{\value{lang}=2}{Доказательство}{Доведення}}}

\begin{center}\small{1. INTRODUCTION}\end{center}

Let a function $f$ be non-negative on a bounded interval $I\subset\mathbb{R}.$ For fixed $\alpha\neq0$ denote by $M_{I,\alpha}(f)$ the means of $f$
$$
M_{I,\alpha}(f)=\left({\frac{1}{|I|}}
{\int_I{f^\alpha{(x)}}\,dx}\right)^\frac{1}{\alpha},
$$
where $|\cdot|$ refers to the Lebesgue measure.

The means $M_{I,\alpha}(f)$ increase as $\alpha$ increases [1, p.144]:
according to the H\"older inequality, for $\alpha<\beta$ and $f\not\sim const$ on the interval $I$ the relation $M_{I,\alpha}(f)<M_{I,\beta}(f)$ holds. This inequality remains valid in the case $\alpha\beta=0$ as well, but $M_{I,0}$ is defined in another way. Throughout the paper we will assume $\alpha\beta\ne0.$

We consider functions satisfying the Reverse H\"older Inequality, i.e., the class of functions $f$ such that
\begin{equation*}
P_{\alpha,\beta}(f)\equiv\sup_{I\subset\mathbb{R_{+}}}
{\frac{M_{I,\beta}(f)}{M_{I,\alpha}(f)}}<+\infty,
\end{equation*}
where the supremum is taken over all intervals $I$ from the positive real axis,
and research on the extension of this condition on the whole real axis. In other words,
we compare $P_{\alpha,\beta}(f)$ with
\begin{equation}\label{02}
{R_{\alpha,\beta}(\overline{f})}\equiv\sup_{I\subset\mathbb{R}}
{\frac{M_{I,\beta}(\overline{f})}{M_{I,\alpha}(\overline{f})}},
\end{equation}
where $I$ stands for different intervals and $\overline{f}$ denotes the even extension of a function $f$ on the real axis.

For the class $ \mathcal A$ of all arbitrary non-negative on $(0;+\infty)$ functions the task is to estimate the constant
$$A_{\alpha,\beta}=\sup_{f\in\mathcal A}
\frac{R_{\alpha,\beta}(\overline{f})}{P_{\alpha,\beta}(f)}$$
that expresses the measure of distinction between $P_{\alpha,\beta}(f)$ and $R_{\alpha,\beta}(\overline{f}).$ The upper estimate $\overline{A}_{\alpha,\beta}$ of $A_{\alpha,\beta}$ has been obtained in the present work. However, we consider mainly a particular case related to the class of all power functions $f(x)=x^\gamma$ defined on $\mathbb{R_+}$ that is embedded in $ \mathcal A$. In order for the means to be finite and positive we suppose
$$\gamma\in\Gamma_{\alpha,\beta}=\{\gamma\in\mathbb{R}:\;\alpha\gamma>-1,\,\beta\gamma>-1\}=
\left\{\begin{array}{ll}
\left(-\frac1\beta; +\infty\right)\;\textnormal{if}\,0<\alpha<\beta,\\
\left(-\infty; -\frac1\alpha\right)\;\textnormal{if}\,\alpha<\beta<0,\\
\left(-\frac{1}{\beta};-\frac1\alpha\right)\;\textnormal{if}\,\alpha<0<\beta.\\
\end{array}
\right.$$
The value analogous to $A_{\alpha,\beta}$ in the particular case is
\begin{equation*}
C_{\alpha,\beta}=\sup_{\small \begin{array}{cc}f(x)=x^\gamma,\\ \small \gamma\in \Gamma_{\alpha,\beta}\end{array}}\frac{R_{\alpha,\beta}(\overline{f})}{P_{\alpha,\beta}(f)}.
\end{equation*}
We calculate the precise value of $C_{\alpha,\beta}$ that can be considered the lower estimate for $A_{\alpha,\beta}.$ Moreover, we obtain the asymptotical equality of the lower and upper estimates $C_{\alpha,\beta}$ and $\overline{A}_{\alpha,\beta}$ respectively.

\begin{center}\small{2. GENERAL ESTIMATES}\end{center}

Let us obtain first a simple upper estimate of $A_{\alpha,\beta}.$

\begin{theorem}\label{22}
The following relation holds
\begin{equation*}
A_{\alpha,\beta}\leq \overline{A}_{\alpha,\beta}\equiv\left\{\begin{array}{ll}
2^{\frac1\alpha}\quad\textnormal{if}\quad0<\alpha<\beta,\\
2^{-\frac1\beta}\quad\textnormal{if}\quad\alpha<\beta<0,\\
2^{\frac1\beta-\frac1\alpha}\quad\textnormal{if}\quad\alpha<0<\beta.\\
\end{array}
\right.
\end{equation*}
\end{theorem}

\begin{proof} Let $f\in\mathcal A.$ For its even extension $\overline{f}$ it is enough to take the supremum in (\ref{02}) only over the intervals $(-a;b),$ where $a=\varepsilon b,$ $\varepsilon\in[0;1].$

Therefore, if $0<\alpha<\beta$,
$$
R_{\alpha,\beta}(\overline{f})=\sup_{b\in\mathbb{R}_+,\,0\leq\varepsilon\leq1}{\frac
{\left(\frac{1}{b(1+\varepsilon)}\int_{-\varepsilon
b}^b{\overline{f}^{\beta}\,dx}\right)^\frac1\beta}
{\left(\frac{1}{b(1+\varepsilon)}\int_{-\varepsilon
b}^b{\overline{f}^{\alpha}\,dx}\right)^\frac1\alpha}}\leq
\sup_{b\in\mathbb{R}_+,\,0\leq\varepsilon\leq1}{\frac
{\left(\frac{1}{b(1+\varepsilon)}\int_{-b}^b{\overline{f}^{\beta}\,dx}\right)^\frac1\beta}
{\left(\frac{1}{b(1+\varepsilon)}\int_{0}^b{\overline{f}^{\alpha}\,dx}\right)^\frac1\alpha}}=
$$
$$
=\sup_{b\in\mathbb{R}_+,\,0\leq\varepsilon\leq1}\left(\left(\frac1{1+\varepsilon}\right)^{\frac1\beta-\frac1\alpha}\cdot
\frac{2^{\frac1\beta}{\left(\frac{1}{b}\int_0^b{f^{\beta}\,dx}\right)^\frac1\beta}}
{\left(\frac{1}{b}\int_0^b{f^{\alpha}\,dx}\right)^\frac1\alpha}\right)=
$$
$$
={2^{\frac1\beta}}\max_{0\leq\varepsilon\leq1}\left(\frac1{1+\varepsilon}\right)^{\frac1\beta-\frac1\alpha}\cdot
P_{\alpha,\beta}(f)={2^{\frac1\alpha}}P_{\alpha,\beta}(f),
$$
that implies $A_{\alpha,\beta}\leq{2^{\frac1\alpha}}.$

If $\alpha<\beta<0$, it can be derived in a similar way that
$$
R_{\alpha,\beta}(\overline{f})={2^{-\frac1\beta}}P_{\alpha,\beta}(f),
$$
and this is followed by the estimate
$A_{\alpha,\beta}\leq{2^{-\frac1\beta}}.$

Eventually, in the case $\alpha<0<\beta$ we obtain

$$
R_{\alpha,\beta}(f)=\sup_{b\in\mathbb{R}_+,\,0\leq\varepsilon\leq1}{\frac
{\left(\frac{1}{b(1+\varepsilon)}\int_{-\varepsilon
b}^b{\overline{f}^{\beta}\,dx}\right)^\frac1\beta}
{\left(\frac{1}{b(1+\varepsilon)}\int_{-\varepsilon
b}^b{\overline{f}^{\alpha}\,dx}\right)^\frac1\alpha}}\leq
\sup_{b\in\mathbb{R}_+,\,0\leq\varepsilon\leq1}{\frac
{\left(\frac{1}{b(1+\varepsilon)}\int_{-b}^b{\overline{f}^{\beta}\,dx}\right)^\frac1\beta}
{\left(\frac{1}{b(1+\varepsilon)}\int_{-b}^b{\overline{f}^{\alpha}\,dx}\right)^\frac1\alpha}}=
$$
$$
={2^{\frac{1}{\beta}-\frac1\alpha}}\max_{0\leq\varepsilon\leq1}\left(\frac1{1+\varepsilon}\right)^{\frac1\beta-\frac1\alpha}\cdot
P_{\alpha,\beta}(f)={2^{\frac{1}{\beta}-\frac1\alpha}}P_{\alpha,\beta}(f)
$$
that is followed by the estimate $A_{\alpha,\beta}\leq{2^{\frac1\beta-\frac1\alpha}}.$

Combining these three cases together completes the proof.
\end{proof}

Though the derived estimate is rather simple, it will be shown that it is asymptotically sharp.

Now let us focus on obtaining the lower estimate $C_{\alpha,\beta}$ for $A_{\alpha,\beta}$ by considering only power functions among all functions contained in the class $ \mathcal A.$ We first prove the next auxiliary statement that is useful to simplify the process of calculating $P_{\alpha,\beta}(f)$ in the case of a monotone function $f\in\mathcal A.$

\begin{theorem}
\label{04}
Let $\alpha<\beta$ and let $f$ be a non-negative monotone function on $\mathbb{R_+}.$ Assume $f^\alpha$ and $f^\beta$ are summable on every interval $I\subset\mathbb{R_+}.$ Then
$$
P_{\alpha,\beta}(f)=
\sup_{(0;\varepsilon)\subset\mathbb{R_{+}}}
{\frac{M_{(0;\varepsilon),\beta}(f)}{M_{(0;\varepsilon),\alpha}(f)}}.
$$
\end{theorem}

\begin{proof}
It is enough to prove that for any interval $I\subset\mathbb{R_{+}}$ there exists $\varepsilon>0$ such that
\begin{equation}\label{35}
{\frac{M_{I,\beta}(f)}{M_{I,\alpha}(f)}}\leq
{\frac{M_{(0;\varepsilon),\beta}(f)}{M_{(0;\varepsilon),\alpha}(f)}}.
\end{equation}

Setting $g=f^\alpha$ when $0<\alpha<\beta$ and $g=f^\beta$ as $\alpha<\beta<0$ respectively in (\ref{35}), we can obtain the analogous inequality for the function $g$ and $1=\alpha<\beta,$ so that to cover the case $\alpha\beta>0$ it is sufficient to prove (\ref{35}) only for the case $1=\alpha<\beta.$ When $\alpha<0<\beta,$ defining $g=f^\beta,$ we convert (\ref{35}) to the case $\alpha<0<\beta=1,$ hence we can prove (\ref{35}) only on the assumption that $\beta=1$.

Fix an arbitrary interval $I\subset\mathbb{R_+}.$ Because of the monotonicity of $f$ there exist the interval $(0;\varepsilon)\supseteq I$ such that
\begin{equation}\label{32}{\frac{1}{|I|}}{\int_I{f(x)}\,dx}=
{\frac{1}{\varepsilon}}{\int_0^\varepsilon{f(x)}\,dx}.\end{equation}
It is known [2, p.160] that in the case of equality (\ref{32}) for any positive convex downwards function $\varphi$ the inequality
\begin{equation}\label{36}
{\frac{1}{|I|}}\int_I{\varphi(f(x))\,dx}\leq{\frac{1}{\varepsilon}}{\int_0^\varepsilon{\varphi(f(x))}\,dx}
\end{equation}
holds.

For $1=\alpha<\beta$ the required inequality (\ref{35}) follows by combining (\ref{32}) and (\ref{36}), where $\varphi(t)=t^\beta\,(\beta>1).$ If $\alpha<0<\beta=1,$ we set $\varphi(t)=t^\alpha\,(\alpha<0)$ in (\ref{36}) and together with (\ref{32}) it implies (\ref{35}) for such values $\alpha$ and $\beta.$
\end{proof}

Let $f(x)=x^\gamma,\,\gamma\in\Gamma_{\alpha,\beta}.$ In order to compare $P_{\alpha,\beta,\gamma}\equiv P_{\alpha,\beta}(x^\gamma)$ and $R_{\alpha,\beta,\gamma}\equiv R_{\alpha,\beta}(|x|^\gamma),$ first obtain their analytic expressions.

By Theorem~\ref{04}, in order to compute $P_{\alpha,\beta,\gamma}$ it is enough to take the supremum only over the intervals of the form $(0;\varepsilon),\,\varepsilon>0,$ not over all arbitrary intervals $I\subset\mathbb{R_{+}}.$
Moreover, observe that for a power function $f$ the value of the expression on the right-hand side of (\ref{35}) is the same at various $\varepsilon>0$ values and
\begin{equation*}
{P_{\alpha,\beta,\gamma}}=\frac{(\gamma\alpha+1)^\frac1\alpha}{(\gamma\beta+1)^\frac1\beta}.
\end{equation*}

For calculating $R_{\alpha,\beta,\gamma}$ the following lemma is useful.

\begin{lemma}
\label{10}
Let $\gamma\in\Gamma_{\alpha,\beta}.$ Then for every function $\overline{f}(x)=|x|^\gamma$
\begin{equation}\label{11}
\sup_{I\subset\mathbb{R}}{\frac{M_{I,\beta}(\overline{f})}{M_{I,\alpha}(\overline{f})}}=
\sup_{0\leq\varepsilon\leq1}{\frac{M_{(-\varepsilon;1),\beta}(\overline{f})}{M_{(-\varepsilon;1),\alpha}(\overline{f})}}.
\end{equation}
\end{lemma}

\begin{proof}
Since $\overline{f}$ is an even function, in the expression on the left in (\ref{11}) it is enough to take the supremum only over the intervals $I=(-a;b)$ such that $|a|\leq|b|.$ Then, by the change of variable $t=\frac{x}{b}$ in both integrals and setting $\varepsilon=\frac{a}{b}$ we get exactly the right-hand side of (\ref{11}).
\end{proof}

A straightforward computation together with (\ref{11}) gives
$$
R_{\alpha,\beta,\gamma}=\sup_{0\leq\varepsilon\leq1}
\left(\frac{{\left(\varepsilon^{\gamma\beta+1}+1\right)^\frac1\beta}
{\left(1+\varepsilon\right)^\frac1\alpha}}{{\left(\varepsilon^{\gamma\alpha+1}+1\right)^\frac1\alpha}
{\left(1+\varepsilon\right)^\frac1\beta}}\cdot P_{\alpha,\beta,\gamma}\right),
$$
or, equivalently,
\begin{equation}\label{12}
R_{\alpha,\beta,\gamma}=\left(\max_{0\leq\varepsilon\leq1}C_{\alpha,\beta,\gamma}(\varepsilon)\right)\cdot
P_{\alpha,\beta,\gamma},
\end{equation}
where
\begin{equation}\label{07}
C_{\alpha,\beta,\gamma}(\varepsilon)=\frac{{\left(\varepsilon^{\gamma\beta+1}+1\right)^\frac1\beta}
{\left(1+\varepsilon\right)^\frac1\alpha}}{{\left(\varepsilon^{\gamma\alpha+1}+1\right)^\frac1\alpha}
{\left(1+\varepsilon\right)^\frac1\beta}}.
\end{equation}

Denote by $C_{\alpha,\beta,\gamma}$ the maximum of the function $C_{\alpha,\beta,\gamma}(\varepsilon)$ on $[0;1].$
This function is continuous on $[0;1]$ for all $\alpha,\beta,$ and $\gamma$, hence, by the extreme value theorem, the maximum on $[0;1]$ exists (and replacing the supremum by the maximum in the expression (\ref{12}) for $R_{\alpha,\beta,\gamma}$ is correct) and is attained at some $\varepsilon=\varepsilon^0_{\alpha,\beta,\gamma}\in[0;1].$ According to the necessary condition for a local extremum,  $\varepsilon^0_{\alpha,\beta,\gamma}$ is the solution of the equation
$$
C\,'_{\alpha,\beta,\gamma}(\varepsilon)=0
$$
as well as the equivalent one
\begin{equation}\label{13}
(\alpha-\beta)\left(\varepsilon^{\gamma\alpha+\gamma\beta+1}-1\right)+\beta(\gamma\alpha+1)
\left(\varepsilon^{\gamma\beta+1}-\varepsilon^{\gamma\alpha}\right)+\alpha(\beta\gamma+1)\left(\varepsilon^{\gamma\beta}-\varepsilon^{\gamma\alpha+1}\right)=0.
\end{equation}

The equation (\ref{13}) clearly has the solution $\varepsilon=1,$ but this point is not the maximum point of $C_{\alpha,\beta,\gamma}(\varepsilon)$ on $[0;1].$ Indeed, $C_{\alpha,\beta,\gamma}(1)=1$ and it will be noted in Remark~\ref{37} on Lemma~\ref{18} that for any $\varepsilon\in(0;1)$ the relation $C_{\alpha,\beta,\gamma}(\varepsilon)>1$ holds. Therefore, $C_{\alpha,\beta,\gamma}>1.$

The constant $C_{\alpha,\beta,\gamma}$ reflects the relation between $P_{\alpha,\beta,\gamma}$ and $R_{\alpha,\beta,\gamma},$ as according to (\ref{12})
$$
\frac{R_{\alpha,\beta,\gamma}}{P_{\alpha,\beta,\gamma}}=
C_{\alpha,\beta,\gamma}.
$$
The explicit expression for the $C_{\alpha,\beta,\gamma}$ is difficult to find as this task is associated with solving the equation (\ref{13}), so that we will focus on estimating $C_{\alpha,\beta,\gamma}.$ Consider different cases of the values of $\alpha,\beta,$ and $\gamma.$

\textbf{A.} Let $0<\alpha<\beta,$ $\gamma\in\Gamma_{\alpha,\beta}$.

If $\gamma\geq0,$ directly majorizing the function $C_{\alpha,\beta,\gamma}(\varepsilon),$ we can get an upper estimate for $C_{\alpha,\beta,\gamma}$ that is more precise than $\overline{A}_{\alpha,\beta}$. Indeed, in this case for all $\varepsilon\in[0;1]$
$$
C_{\alpha,\beta,\gamma}(\varepsilon)={\left(1+\varepsilon\right)^{\frac1\alpha-\frac1\beta}}\frac{\left(\varepsilon^{\gamma\beta+1}+1\right)^\frac1\beta}
{\left(\varepsilon^{\gamma\alpha+1}+1\right)^\frac1\alpha}\leq{\left(1+\varepsilon\right)^{\frac1\alpha-\frac1\beta}}\frac{\left(\varepsilon^{\gamma\alpha+1}+1\right)^\frac1\beta}
{\left(\varepsilon^{\gamma\alpha+1}+1\right)^\frac1\alpha}\leq{\left(1+\varepsilon\right)^{\frac1\alpha-\frac1\beta}}
$$
and so
\begin{equation}\label{06}
C_{\alpha,\beta,\gamma}(\varepsilon)\leq{\left(1+\varepsilon\right)^{\frac1\alpha-\frac1\beta}}\quad
\textnormal{and}\quad C_{\alpha,\beta,\gamma}\leq2^{\frac1\alpha-\frac1\beta}\quad(\gamma\geq0)
\end{equation}
follow.

Analogously, in the case $-\frac1\beta<\gamma<0$ we have
$$
C_{\alpha,\beta,\gamma}(\varepsilon)={\left(1+\varepsilon\right)^{\frac1\alpha-\frac1\beta}}\frac{\left(\varepsilon^{\gamma\beta+1}+1\right)^\frac1\beta}
{\left(\varepsilon^{\gamma\alpha+1}+1\right)^\frac1\alpha}\leq{\left(1+\varepsilon\right)^{\frac1\alpha-\frac1\beta}}\frac{2^\frac1\beta}
{\left(\varepsilon+1\right)^\frac1\alpha}\leq{\left(\frac{2}{\varepsilon+1}\right)^{\frac1\beta}}
$$
that implies
\begin{equation}\label{17}
C_{\alpha,\beta,\gamma}(\varepsilon)\leq{\left(\frac{2}{\varepsilon+1}\right)^{\frac1\beta}}
\quad\textnormal{and}\quad C_{\alpha,\beta,\gamma}\leq2^{\frac1\beta}\quad(\gamma<0).
\end{equation}

Further, let us formulate the following auxiliary statement. It shows that as $\alpha$ and $\beta$ are fixed, the graphs of the functions $C_{\alpha,\beta,\gamma}(\varepsilon)$ that correspond to different values of $\gamma$ do not have any intersection points in the interval $(0;1).$

\begin{lemma}
\label{18}
Let $0<\alpha<\beta$ and
$\gamma_1,\,\gamma_2\in\Gamma_{\alpha,\beta}$ are such that $\gamma_1<\gamma_2.$ Then if $\gamma_1\geq0$ ($\gamma_2\leq0,$ respectively), for all $\varepsilon\in[0;1]$ the relation $C_{\alpha,\beta,\gamma_1}(\varepsilon)\leq C_{\alpha,\beta,\gamma_2}(\varepsilon)$ ($C_{\alpha,\beta,\gamma_1}(\varepsilon)\geq C_{\alpha,\beta,\gamma_2}(\varepsilon)$, respectively) holds. Moreover, the equality takes place only at the ends of the interval $[0;1].$
\end{lemma}

\begin{proof}
Consider the case $0\leq\gamma_1<\gamma_2$ first. The inequality
$$C_{\alpha,\beta,\gamma_1}(\varepsilon)<C_{\alpha,\beta,\gamma_2}(\varepsilon),\,\varepsilon\in(0;1),$$
according to (\ref{07}), is equivalent to the following one:
\begin{equation}\label{19}
\frac{\left(\varepsilon^{\gamma_1\beta+1}+1\right)^\frac1\beta}
{\left(\varepsilon^{\gamma_1\alpha+1}+1\right)^\frac1\alpha}<
\frac{\left(\varepsilon^{\gamma_2\beta+1}+1\right)^\frac1\beta}
{\left(\varepsilon^{\gamma_2\alpha+1}+1\right)^\frac1\alpha}.
\end{equation}

The inequality (\ref{19}) holds since the function
$$\psi(x)=\frac{\left(\varepsilon^{x\beta+1}+1\right)^\frac1\beta}
{\left(\varepsilon^{x\alpha+1}+1\right)^\frac1\alpha},\;x\in(0;1)$$
increases strictly on $[0;+\infty).$ Indeed,
$$\psi'(x)=\frac{\left(\varepsilon^{x\beta+1}+1\right)^{\frac1\beta-1}}
{\left(\varepsilon^{x\alpha+1}+1\right)^{\frac1\alpha+1}}\cdot\ln{\varepsilon}
\cdot\left(\varepsilon^{x\beta+1}-\varepsilon^{x\alpha+1}\right)>0,\quad
x>0.$$

Furthermore, as $\psi'(x)<0$ when $x<0,$ the function $\psi=\psi(x)$ is strictly decreasing on $(-\infty;0],$ and in the case of $\gamma_1<\gamma_2\leq0$ the inequality reverse to (\ref{19}) holds. The equality of the functions $y=C_{\alpha,\beta,\gamma_1}(\varepsilon)$ and $y=C_{\alpha,\beta,\gamma_2}(\varepsilon)$ at the ends of $[0;1]$ can be checked by direct calculation.
\end{proof}

\begin{remark}
\label{37}\; Lemma~\ref{18} in particular implies $C_{\alpha,\beta,\gamma}(\varepsilon)>1$ for all $\varepsilon\in(0;1)$ and arbitrary values of $\alpha,\beta,$ and $\gamma.$ \textnormal{Indeed, in the case $\gamma>0$ it is enough to set $\gamma_1=0,\,\gamma_2=\gamma$ in (\ref{19}) and divide the inequality by its left-hand side. If $\gamma<0,$ set $\gamma_1=\gamma,\,\gamma_2=0$ and use the opposite to (\ref{19}) inequality analogously.}
\end{remark}

Let us obtain the general estimate of
$$C_{\alpha,\beta}=\sup_{\gamma>-\frac1\beta}{C_{\alpha,\beta,\gamma}}=\max\{C^{+}_{\alpha,\beta},C^{-}_{\alpha,\beta}\},$$
where
$C^{+}_{\alpha,\beta}=\sup_{\gamma\geq0}{C_{\alpha,\beta,\gamma}},$
$C^{-}_{\alpha,\beta}=\sup_{-\frac{1}{\beta}<\gamma<0}{C_{\alpha,\beta,\gamma}}.$

If $\gamma\geq0,$ according to Lemma~\ref{18}, the curve $C_{\gamma}(\varepsilon)\equiv C_{\alpha,\beta,\gamma}(\varepsilon)$ corresponding to a larger value of $\gamma$ is above all the graphs that are related to smaller values of $\gamma.$
Moreover, the majorant of $C_{\gamma}(\varepsilon),\,\gamma\geq0$ given in (\ref{06}), i.e., ${\left(\varepsilon+1\right)^{\frac1\alpha-\frac1\beta}},$ together with
\begin{equation*}
C_{\alpha,\beta,\gamma}=\max_{0\leq
\varepsilon\leq1}\left(\frac{{\left(\varepsilon^{\gamma\beta+1}+1\right)^\frac1\beta}
{\left(1+\varepsilon\right)^\frac1\alpha}}{{\left(\varepsilon^{\gamma\alpha+1}+1\right)^\frac1\alpha}
{\left(1+\varepsilon\right)^\frac1\beta}}\right)\rightarrow\max_{0\leq
\varepsilon\leq1}
{\left(\varepsilon+1\right)^{\frac1\alpha-\frac1\beta}}=2^{\frac1\alpha-\frac1\beta}
\end{equation*}
that holds as $\gamma~\rightarrow~+\infty$ implies $C^{+}_{\alpha,\beta}=2^{\frac1\alpha-\frac1\beta}.$

If $\gamma\in(-\frac{1}{\beta};0),$ by Lemma \ref{18} the curve $C_{\gamma}(\varepsilon)$ that corresponds to a smaller value of $\gamma$ is above all the graphs that are related to larger values of $\gamma.$
Further, according to (\ref{17}), the functions $C_{\gamma}(\varepsilon),\,\gamma\in(-\frac{1}{\beta};0)$ are bounded above by ${\left(\frac2{\varepsilon+1}\right)^{\frac1\beta}}.$ More precisely,

$$
C_{\alpha,\beta,\gamma}(\varepsilon)\rightarrow\max_{0\leq \varepsilon\leq1}
{\left(\frac{\varepsilon+1}{\varepsilon^{\frac{\beta-\alpha}{\beta}}+1}\right)^{\frac1\alpha}}
{\left(\frac2{\varepsilon+1}\right)^{\frac1\beta}}={2^{\frac1\beta}},
$$
as $\gamma\rightarrow{-\frac1{\beta}}$ that is followed by $C^{-}_{\alpha,\beta}=2^{\frac1\beta}.$

Thus, the value
$$C_{\alpha,\beta}=\max\left\{2^{\frac1\alpha-\frac1\beta},2^{\frac1\beta}\right\}$$
depends on the relationship between $\alpha$ and $\beta$ and
\begin{equation}\label{25}
C_{\alpha,\beta}= \left\{\begin{array}{ll}
2^{\frac1\alpha-\frac1\beta}\quad\textnormal{if}\quad 0<\alpha\leq\frac{\beta}{2},\\
2^{\frac1\beta}\quad\textnormal{if}\quad\frac{\beta}{2}<\alpha<\beta\\
\end{array}
\right.
\end{equation}
holds.

\textbf{B.} Let $\alpha<\beta<0,$ $\gamma\in\Gamma_{\alpha,\beta}$.

As in the case A direct estimations of $C_{\alpha,\beta,\gamma}(\varepsilon)$ imply the following: for all $\varepsilon\in[0;1]$
\begin{equation}\label{26}
C_{\alpha,\beta,\gamma}(\varepsilon)\leq{\left(1+\varepsilon\right)^{\frac1\alpha-\frac1\beta}}\quad
\textnormal{and}\quad C_{\alpha,\beta,\gamma}\leq2^{\frac1\alpha-\frac1\beta}\quad(\gamma<0)
\end{equation}
and
\begin{equation}\label{27}
C_{\alpha,\beta,\gamma}(\varepsilon)\leq{\left(\frac{\varepsilon+1}{2}\right)^{\frac1\alpha}}\quad
\textnormal{and}\quad C_{\alpha,\beta,\gamma}\leq2^{-\frac1\alpha}\quad\left(0\leq\gamma<-\frac1\alpha\right).
\end{equation}

It is easy to show that in the case $\alpha<\beta<0$ Lemma \ref{18} remains valid. In order to do this it is enough to repeat exactly the same proof.

Analogously to the case A, the relations
$$
C_{\alpha,\beta,\gamma}\rightarrow\max_{0\leq\varepsilon\leq1}
{\left(\varepsilon+1\right)^{\frac1\alpha-\frac1\beta}}=2^{\frac1\alpha-\frac1\beta}
$$
as $\gamma~\rightarrow~{-\infty},\;\gamma<0$ and
$$
C_{\alpha,\beta,\gamma}
\rightarrow\max_{0\leq \varepsilon\leq1}
{\left(\frac{\varepsilon^{\frac{\alpha-\beta}{\alpha}}+1}{\varepsilon+1}\right)^{\frac1\beta}}
{\left(\frac{\varepsilon+1}2\right)^{\frac1\alpha}}={2^{-\frac1\alpha}},
$$
as $\gamma\rightarrow{-\frac1\alpha},\;\gamma\in[0;-\frac{1}{\alpha})$
together with the majorants (\ref{26}) and (\ref{27}) imply

\begin{equation}\label{28}
C_{\alpha,\beta}= \left\{\begin{array}{ll}
2^{\frac1\alpha-\frac1\beta}\quad\textnormal{if}\quad\alpha\leq2\beta,\\
2^{-\frac1\alpha}\quad\textnormal{if}\quad 2{\beta}<\alpha<\beta.\\
\end{array}
\right.
\end{equation}

\textbf{C.} Let $\alpha<0<\beta,$ $\gamma\in\Gamma_{\alpha,\beta}.$

Directly majorizing the function $C_{\alpha,\beta,\gamma}(\varepsilon)$ as before, for all $\varepsilon\in[0;1]$ we obtain
\begin{equation}\label{29}
C_{\alpha,\beta,\gamma}(\varepsilon)\leq{\left(\frac{\varepsilon+1}{2}\right)^{\frac1\alpha}}
\quad\textnormal{and}\quad C_{\alpha,\beta,\gamma}\leq2^{-\frac1\alpha}\quad\left(0\leq\gamma<-\frac1\alpha\right).
\end{equation}
and
\begin{equation}\label{30}
C_{\alpha,\beta,\gamma}(\varepsilon)\leq{\left(\frac{2}{\varepsilon+1}\right)^{\frac1\beta}}
\quad\textnormal{and}\quad C_{\alpha,\beta,\gamma}\leq2^{\frac1\beta}\quad\left(-\frac1\beta<\gamma<0\right).
\end{equation}

It can be showed that in the case $\alpha<0<\beta$ Lemma \ref{18} remains valid.

Since
$$
C_{\alpha,\beta,\gamma}
\rightarrow\max_{0\leq \varepsilon\leq1}
{\left(\frac{\varepsilon+1}{\varepsilon^{\frac{\beta-\alpha}{\beta}}+1}\right)^{\frac1\alpha}}
{\left(\frac2{\varepsilon+1}\right)^{\frac1\beta}}={2^{\frac1\beta}}
$$
as $\gamma\rightarrow-\frac1\beta,\; -\frac1\beta<\gamma<0$ and
$$
C_{\alpha,\beta,\gamma}
\rightarrow\max_{0\leq \varepsilon\leq1}
{\left(\frac{\varepsilon^{\frac{\alpha-\beta}{\alpha}}+1}{\varepsilon+1}\right)^{\frac1\beta}}
{\left(\frac{\varepsilon+1}2\right)^{\frac1\alpha}}={2^{-\frac1\alpha}},
$$
as $\gamma\rightarrow{-\frac1\alpha},\;0\leq\gamma<-\frac1\alpha,$
taking into account (\ref{29}) and (\ref{30}), derive

\begin{equation}\label{31}
C_{\alpha,\beta}= \left\{\begin{array}{ll}
2^{\frac1\beta}\quad\textnormal{if}\quad 0<\beta\leq{-\alpha},\\
2^{-\frac1\alpha}\quad\textnormal{if}\quad\beta>-\alpha\\
\end{array}
\right.=2^{1/{\min\{|\alpha|,\beta\}}}.
\end{equation}

\begin{center}\small{3. CONCLUSION}\end{center}

Let us combine all the results obtained in the cases A, B, and C.
It follows from (\ref{25}), (\ref{28}), and (\ref{31}) that for different values of $\alpha, \beta,$ and $\gamma$ the strict inequality $C_{\alpha,\beta}<\overline{A}_{\alpha,\beta}$ holds. The left-hand and the right-hand side of this inequality are a lower and an upper estimate for $A_{\alpha,\beta}$ respectively. Because of the strict inequality sign we cannot derive the precise value of $A_{\alpha,\beta}.$ However, $C_{\alpha,\beta}$ is asymptotically equal to $\overline{A}_{\alpha,\beta}.$ Indeed, it follows directly from the analytical expressions of $C_{\alpha,\beta}$ for each case of $\alpha$ and $\beta$ values that

$$C_{\alpha,\beta}\sim\overline{A}_{\alpha,\beta}=2^\frac1\alpha,\,\beta\rightarrow +\infty\quad\textnormal{when}
\quad0<\alpha<\beta;$$
$$C_{\alpha,\beta}\sim\overline{A}_{\alpha,\beta}=2^{-\frac1\beta},\,\alpha\rightarrow -\infty\quad\textnormal{when}
\quad\alpha<\beta<0;$$
$$C_{\alpha,\beta}=2^{-\frac1\alpha}\sim\overline{A}_{\alpha,\beta},\,\beta\rightarrow +\infty$$
or
$$C_{\alpha,\beta}=2^{\frac1\beta}\sim\overline{A}_{\alpha,\beta},\,\alpha\rightarrow -\infty\quad\textnormal{when}\quad
\alpha<0<\beta.$$

Thus, as $\alpha$ and $\beta$ are fixed $A_{\alpha,\beta}\in\left[C_{\alpha,\beta};\overline{A}_{\alpha,\beta}\right],$ but this interval has nonzero measure so that does not determine $A_{\alpha,\beta}$ precisely. In the limiting cases the relations above set equivalency of the lower and the upper estimates of $A_{\alpha,\beta}.$
{\footnotesize

\end{document}